\documentclass[12pt,reqno]{amsart}
\pdfoutput=1
\usepackage[headings]{fullpage}
\usepackage{amsfonts}
\usepackage{amsthm}
\usepackage{amssymb}
\usepackage{epstopdf}
\usepackage{graphicx}
\usepackage{texdraw}
\usepackage{url}
\usepackage[all]{xy}
\usepackage{float}   
\usepackage[bookmarks=true,%
    colorlinks=true,%
    linkcolor=blue,%
    citecolor=blue,%
    filecolor=blue,%
    menucolor=blue,%
    urlcolor=blue,%
    breaklinks=true]{hyperref}

\newtheorem{theorem}{Theorem}[section]
\theoremstyle{definition}
\newtheorem{proposition}[theorem]{Proposition}
\newtheorem{lemma}[theorem]{Lemma}
\newtheorem{definition}[theorem]{Definition}
\newtheorem{remark}[theorem]{Remark}
\newtheorem{corollary}[theorem]{Corollary}

\newtheorem{example}[theorem]{Example}

\def\BN{\mathbb N}
\def\BZ{\mathbb Z}

\def\BQ{\mathbb Q}

\def\calF{\mathcal F}

\def\la{\langle}
\def\ra{\rangle}

\def\al{\alpha}
\def\ve{\varepsilon}
\def\be { \begin{equation} }
\def\ee { \end{equation} }
\def\sL{\mathsf{L}}
\def\sM{\mathsf{M}}
\def\Ann{\mathrm{Ann}}
\def\GL{\mathrm{GL}}

\def\Zt{\mathsf{R}}

\def\cW{\mathbb W}
\def\cT{\mathbb W} 

\def\cR{\mathsf{R}}
\def\ot{\otimes}

\def\Ext{\operatorname{Ext}}
\def\Trp{{\cT_{r,+}}}
\def\Tr{{\cT_r}}
\def\codim{\mathrm{codim}}

\def\cF{\mathcal F}
\def\rk{\mathrm{dim}}
\def\fM{\mathfrak M}
\def\fm{{\mathfrak m}}
\def\bk{\mathbf k}
\newcommand\no[1]{ }
\def\hbt{\hat \boxtimes}
\def\str{\mathrm{str}}
\def\tI{\tilde I}

\def\tg{\tilde g}
\def\cL{{\mathcal L}}
\def\TrL{\mathbb W_{r,\cL}}
\def\TrLi{\mathbb W_{r,\cL_i}}

\begin{document}

\title[A survey of $q$-holonomic functions]{A survey of $q$-holonomic
functions}

\author{Stavros Garoufalidis}
\address{School of Mathematics \\
         Georgia Institute of Technology \\
         Atlanta, GA 30332-0160, USA \newline
         {\tt \url{http://www.math.gatech.edu/~stavros}}}
\email{stavros@math.gatech.edu}

\author[Thang  T. Q. L\^e]{Thang  T. Q. L\^e}
\address{School of Mathematics \\
         Georgia Institute of Technology \\
         Atlanta, GA 30332-0160, USA \newline
         {\tt \url{http://www.math.gatech.edu/~letu}}}
\email{letu@math.gatech.edu}


\date{September 20, 2016}


\begin{abstract}
We give a survey of basic facts of $q$-holonomic functions of one or 
several variables, following Zeilberger and Sabbah. We provide
detailed proofs and examples.
\end{abstract}

\maketitle

\tableofcontents


\section{Introduction}
\label{sec.intro}

In his seminal paper~\cite{Z90} Zeilberger introduced the class of holonomic
functions (in several discrete or continuous variables), and proved that
it is closed under several operations (including sum and product).
Zeilberger's main theorem asserts that combinatorial identities
of multivariable binomial sums can be proven automatically, by exhibiting
a certificate of a recursion for such sums, and by checking a finite
number of initial conditions. Such a recursion is guaranteed within the
class of holonomic functions, and an effective certificate can be
computed by Zeilberger's telescoping methods~\cite{Z90,WZ}.
Numerous examples of this philosophy were given in the book~\cite{PWZ}.

Holonomic sequences of one variable are those that
satisfy a linear recursion with polynomial coefficients. Holonomic sequences
of two (or more variables) also satisfy a linear recursion with polynomial
coefficients with respect to each variable, but they usually satisfy
additional linear recursions that form a maximally overdetermined system.
The precise definition of holonomic functions requires a theory of
dimension (developed using homological algebra) and a key Bernstein
inequality.

Extending Wilf-Zeilberger's class of holonomic functions to the class
of $q$-holonomic functions is by no means obvious, and was achieved by
Sabbah~\cite{Sab}. Sabbah's article was written using the 
language of homological algebra.

The distance between Zeilberger's and
Sabbah's papers is rather large: the two papers were written for different
audiences and were read by largely disjoint audiences. The purpose of our
paper is to provide a bridge between Zeilberger's and Sabbah's
paper, and in particular to translate Sabbah's article into the class
of multivariate functions. En route, we decided to give a self-contained
survey (with detailed proofs and examples) of basic properties of
$q$-holonomic functions of one or several variables.
We claim no originality of the results presented here, except perhaps
of a proof that multisums of $q$-holonomic functions are $q$-holonomic,
in all remaining variables (Theorem \ref{thm.multisum}).
This last property is
crucial for $q$-holonomic functions that arise naturally in quantum topology.
In fact, quantum knot invariants, such as  the colored Jones
polynomial of a knot or link (colored by irreducible representations
of a simple Lie algebra), and the HOMFLY-PT polynomial of a link, colored by
partitions with a fixed number of rows are multisums of $q$-proper
hypergeometric functions~\cite{GL,GLL}. Therefore, they are $q$-holonomic
functions.

We should point out a difference in how recurrence relations are viewed
in quantum topology versus in combinatorics. In the former, minimal order
recurrence relations often have geometric meaning,
and in the case of the Jones or HOMFLY-PT polynomial of
a knot, is conjectured to be a deformation of the character variety
of the link complement~\cite{Ga:AJ,Le,Le2,LT,LZ}. In the latter, 
recurrence relations are used as a convenient way to automatically prove 
combinatorial identities.

Aside from the geometric interpretation of a recurrence for the colored
Jones polynomial of a knot, and for the natural problem of
computing or guessing such recursions, we should point out that such
recursions can also be used to numerically compute several terms of the 
asymptotics of
the colored Jones polynomial at complex roots of unity, a fascinating
story that connects quantum topology to hyperbolic geometry and
number theory. For a sample of computations, the reader may
consult~\cite{Ga:arbeit,GK,GK2,GZ}.


\section{$q$-holonomic functions of one variable}
\label{sec:holonomic}

Throughout the paper $\BZ$, $\BN$ and $\BQ$ denotes the sets of integers, 
non-negative integers, and rational numbers respectively. We will fix
a field $\bk$ of characteristic zero, and a variable $q$ transcendental over
$\bk$.   
Let $\cR=\bk(q)$ denote the field of rational functions on a variable $q$ 
with coefficients in $\bk$..

\subsection{Recurrence relations}
\label{sub.recurrence}

One of the best-known sequences of natural numbers is the Fibonacci sequence
$F(n)$ for $n \in \BN$
that satisfies the recurrence relation
$$
F(n+2)-F(n+1)-F(n)=0, \qquad F(0)=0, F(1)=1 \,.
$$
Similarly, one can consider a $q$-version of the Fibonacci sequence $f(n)$
for $n \in \BN$  that satisfies the recurrence relation
$$
f(n+2)-f(n+1)-qf(n)=0, \qquad f(0)=1, f(1)=2 \,.
$$
In that case, $f(n) \in \BZ[q]$ is a sequence of polynomials in $q$
with integer coefficients.

A $q$-holonomic sequence is one that satisfies a nonzero
linear recurrence with coefficients that are polynomials in $q$ and $q^n$.
More precisely, we say that a function $f: \BN \to V$, where  $V$ is an 
$\cR$-vector space, is $q$-holonomic if there exists $d \in \BN$ and 
$a_j(u,v) \in \bk[u,v]$ for $j=0,\dots$ with $a_d \neq 0$ such that 
for all natural numbers $n$ we have
\be
\label{eq.ad0}
a_d(q^n,q) f(n+d) + \dots + a_0(q^n,q) f(n) =0 \,.
\ee

\subsection{Operator form of recurrence relations}
\label{sub.operator}

We can convert the above definition in operator form as follows. 
Let $V$ an $\cR$-vector space. Let $S_{1,+}(V)$ denote the set of functions 
$f: \BN \to V$, and consider the
operators $\sL$ and $\sM$ that act on $S_{1,+}(V)$ by
\be
\label{eq.ML}
(\sL f)(n) = f(n+1), \qquad (\sM f)(n) = q^n f(n) \,.
\ee
It is easy to see that $\sL$ and $\sM$ satisfy the
$q$-commutation relation $\sL \sM= q \sM \sL$. The algebra
$$
\cT_+ := \Zt\la \sM, \sL \ra/(\sL \sM - q \sM \sL)
$$
is called the quantum plane. Equation~\eqref{eq.ad0} can be written in the
form
$$
P f=0, \qquad P=\sum_{j=0}^d a_j(\sM, q)\sL^j \in \cT_+ \,.
$$
Given any $f \in S_{1,+}(V)$, the set
$$
\Ann_+(f)=\{ P \in \cT_+ \, | Pf=0 \}
$$
is a left ideal of $\cT_+$, and the corresponding submodule $M_{f,+}=\cT_+ f$
of $S_{1,+}(V)$ generated by $f$ is cyclic and isomorphic to $\cT_+/\Ann_+(f)$.
In other words, $M_{f,+} \subset S_1(V)$ consists of all functions obtained by
applying a recurrence operator $P \in \cT_+$ to $f$.
Then, we have the following.

\begin{definition}
\label{def.qholo1+}
$f \in S_{1,+}(V)$ is $q$-holonomic if $\Ann_+(f) \neq \{0\}$.
\end{definition}

Before we proceed further, let us give some elementary examples of
$q$-holonomic functions.

\begin{example}
\label{ex.fg1}
\rm{(a)}
The function $f(n)=(-1)^n$ is $q$-holonomic since it satisfies the
recurrence relation
$$
f(n+1) + f(n)=0, \qquad n \in \BN \,.
$$
\rm{(b)} The functions $f(n)=q^n$, $g(n)=q^{n^2}$ and $h(n)=q^{n(n-1)/2}$
are $q$-holonomic since they satisfy the recurrence relations
$$
f(n+1) - q f(n) =0, \qquad g(n+1)-q^{2n+1} g(n)=0, \qquad
h(n+1) - q^n h(n)=0, \qquad 
 n \in \BN \,.
$$
However, the function $n \mapsto q^{n^3}$ is not $q$-holonomic. Indeed, if it
satisfied a recurrence relation, divide it by $h(n)$ and reach
a contradiction.
\newline
\rm{(c)} The delta function
$$
\delta(n)=
\begin{cases}
1 \quad &\text{ if } n = 0 \\
0  & \text{otherwise},
\end{cases}
$$
is $q$-holonomic since it satisfies the recurrence relation
$$
(1-q^n)\delta(n)=0, \qquad n \in \BN \,.
$$
\rm{(d)} The quantum factorial function given by
$(q;q)_n=\prod_{k=1}^n(1-q^k)$ for $n \in \BN$
is $q$-holonomic, since it satisfies the recurrence relation
\be
\label{eq.qfac}
(q;q)_{n+1} -(1-q^{n+1})(q;q)_n =0, \qquad n \in \BN \,.
\ee
\rm{(e)} The inverse quantum factorial function given by
$n \to 1/(q;q)_n$ for $n \in \BN$
is $q$-holonomic, since it satisfies the recurrence relation
$$
(1- q^{n+1}) \frac 1 {(q;q)_{n+1}} - \frac 1{(q;q)_n} =0\,.
$$
\rm{(f)} Suppose $\bk= \BQ(x)$. Define the {\em $q$-Pochhammer symbol} 
$(x;q)_n$, for $n\in \BN$, by
$$
(x;q)_n:= \prod_{k=1}^n (1-x q^{k-1}) \,.
$$
Then the function $ n \mapsto (x;q)_n$ is $q$-holonomic over $\bk$, since 
it satisfies the recurrence relation
$$
(x;q)_{n+1} + (x q^n-1)  (x;q)_{n}=0 \,.
$$
\end{example}

\subsection{Extension to functions defined on the integers}
\label{sub.Zextension}

For technical reasons that have to do with specialization and
linear substitution, it is useful to extend the definition of
$q$-holonomic functions to ones defined on the set of integers. Note that in 
the setting of Section~\ref{sub.operator} where the domain of the function 
$f$ is $\BN$, $\sM$ is invertible, but $\sL$ is not.

When the domain is $\BZ$, the definitions of the previous section extend 
almost naturally, but with some important twists that we will highlight. Let
$S_1(V)$ denote the set of functions $f: \BZ \to V$. The operators
$\sL$ and $\sM$ still act on $S_1(V)$ via~\eqref{eq.ML}, only that now
they are invertible and generate the quantum Weyl algebra
$$
\cT := \Zt\la \sM^{\pm 1}, \sL^{\pm 1}\ra/(\sL \sM - q \sM \sL) \,.
$$
Given $f \in S_1(V)$, we can define
\be
\label{eq.ann}
\Ann(f)=\{P \in \cT \,| Pf=0\}
\ee
and the corresponding cyclic module $M_f:= \cT f \subset S_1(V)$.

\begin{definition}
\label{def.fqholo1}
$f \in S_1(V)$ is $q$-holonomic  if $\Ann(f) \neq \{0\}$.
\end{definition}

\begin{remark}
\label{rem.determines}
An important property of $q$-holonomic functions is that a $q$-holonomic 
$f$ (with domain $\BN$ or $\BZ$) is completely determined by a non-trivial 
recurrence relation and a finite set of values: observe that the leading 
and trailing coefficients of the recurrence relation, being polynomials
in $q$ and $q^n$,  are nonzero for all but finitely many $n$. For such
$n$, we can compute $f(n \pm 1)$ from previous values. It follows that
$f$ is uniquely determined by its restriction on a finite set.
\end{remark}

It is natural to ask what happens to a $q$-holonomic function defined
on $\BN$ when we extend it by zero to a function on $\BZ$. It is instructive
to look at part (d) of Example \ref{ex.fg1}. Consider the extension of
$(q;q)_n$ to the integers defined by $(q;q)_n=0$ for $n<0$. The recurrence
relation~\eqref{eq.qfac} cannot be solved backwards when $n=-1$. Moreover, 
the recurrence relation \eqref{eq.qfac} does not hold for $n=-1$ for 
{\em any} value of $(q;q)_{-1}$.
On the other hand, if we multiply~\eqref{eq.qfac} by $1-q^{n+1}$ (which vanishes
exactly when $n=-1$), then we have a valid recurrence relation
$$
(1-q^{n+1})(q;q)_{n+1} -(1-q^{n+1})^2(q;q)_n =0, \qquad n \in \BZ \,.
$$
This observation generalizes easily to a proof of the following.

\begin{lemma}
\label{lem.qholo1}
\rm{(a)} If $f \in S_1(V)$ is $q$-holonomic and $g \in S_{1,+}(V)$
is its restriction to the natural numbers, then $g$ is $q$-holonomic.
\newline
\rm{(b)} Conversely, if $g \in S_{1,+}(V)$ is $q$-holonomic and
$f \in S_1(V)$ is the 0 extension of $g$ (i.e.,
$f(n)=g(n)$ for $n \in \BN$, $f(n)=0$ for $n<0$), then $f$ is $q$-holonomic.
\end{lemma}
\begin{proof}
(a)  Suppose $Pf=0$, where $P\in \cT$. Then $\sM^a \sL^b Pf=0$ for all 
$a, b \in \BZ$. For big enough integers $a,b$ we have 
$Q=\sM^a \sL^b P\in \cT_+$ and $Qg=0$. This shows $g$ is $q$-holonomic.

(b) Suppose $Qg=0$ where $0\neq Q\in \cT_+$ and $Q$ has degree $d$ in $\sL$. 
Then $Pf=0$, where $P= \left( \prod_{j=1}^d(1- q^j \sM)\right) Q$. Hence, $f$ 
is $q$-holonomic.
\end{proof}
For stronger statements concerning   more general types of extension, see 
Propositions \ref{prop.modify} and \ref{prop.patching}.  

\begin{example}
\label{ex.fg2}
\rm{(a)} The discrete Heaviside function
\be
\label{eq.heavyside}
H(n) =\begin{cases}
 1 \quad &\text{ if } n \ge 0 \\
 0  & \text{otherwise}
\end{cases}
\ee
is $q$-holonomic since it is the 0 extension of a constant function on 
$\BN$. Alternatively, it satisfies the recurrence relation
$$
(1-q^{n+1})H(n+1)-(1-q^{n+1})H(n)=0, \qquad n \in \BZ \,.
$$
\newline
\rm{(b)} The 0 extensions of all the functions in Example \ref{ex.fg1} 
are $q$-holonomic. In particular, the delta function
$$
\delta: \BZ \to \BZ, \quad \delta(n) =\begin{cases}
 1 \quad &\text{ if } n = 0 \\
 0  & \text{otherwise},
\end{cases}
$$
is $q$-holonomic.
\end{example}


\section{$q$-holonomic functions of several variables}
\label{sec.several}

\subsection{Functions of several variables and the
quantum Weyl algebra}

In this section, we extend our discussion to functions of $r$ variables.
One might think that a $q$-holonomic function of several variables is one
that satisfies a recurrence relation with respect to each variable, when
all others are fixed. Although this is not far from true, this is not always
true. Instead, a $q$-holonomic function needs to satisfy additional recurrence
relations to create a maximally overdetermined system of equations.
Let us explain this now.

For a natural number $r$, let $S_r(V)$ be the set of all functions
$f: \BZ^r \to V$ and $S_{r,+}(V)$ the subset of functions with domain $\BN^r$.
For $i=1,\dots, r$ consider the operators $\sL_i$ and $\sM_i$ which act
on functions $f \in S_r(V)$ by
\begin{align}
(\sL_if)(n_1,\dots, n_i, \dots, n_r) &= f(n_1, \dots, n_i+1, \dots , n_r)\\
(\sM_if)(n_1, \dots, n_r)&= q^{n_i} f(n_1, \dots, n_r).
\end{align}
It is clear that $\sL_i$, $\sM_j$ are invertible operators that satisfy
the $q$-commutation relations
\begin{subequations}
\begin{align}
\sM_i \sM_j &= \sM_j \sM_i
\label{eq.w1}\\
\sL_i \sL_j &=
\sL_j \sL_i
\label{eq.w2}\\
\sL_i \sM_j &= q^{\delta_{i,j}} \sM_j \sL_i
\label{eq.w3}
\end{align}
\end{subequations}
for all $i,j=1,\dots,r$. Here $\delta_{i,j}=1$ when $i=j$ and zero otherwise. 
The {\em $r$-dimensional quantum Weyl algebra $\cT_r$} is the
$\Zt$-algebra generated by
$\sL_1^{\pm 1}, \dots, \sL_r^{\pm 1}, \sM_1^{\pm 1}, \dots, \sM_r^{\pm 1}$
subject to the relations \eqref{eq.w1}--\eqref{eq.w3}. 
Then $\cT_r\cong \cT^{\otimes r}$ and is a Noetherian domain.

Given $f \in S_r(V)$, the {\em annihilator} $\Ann(f)$, which is a left 
$\cT_r$-ideal, is defined as in Equation~\eqref{eq.ann}. The corresponding 
cyclic module $M_f$, defined by $M_f = \cT_r f \subset S_r(V)$, is 
isomorphic to $\cT_r/\Ann(f)$.

Informally, $f$ is $q$-holonomic if $M_f \subset S_r(V)$ is as small as
possible, in a certain measure of complexity. In particular, $\Ann(f)$
must contain recurrence relations with respect to each variable $n_i$
(when all other variables are fixed), but this is not sufficient in general.

\subsection{The case of $\Trp$}
\label{sub.caseWrp}

This and the remaining sections follow closely the work of Sabbah~\cite{Sab}.
Let $\cT_{r,+}$ be the subalgebra of $\cT_r$ generated by non-negative
powers of $\sM_j, \sL_j$.
Our aim is to define the dimension of a finitely generated $\Trp$-module,
to recall the Bernstein inequality (due to Sabbah), and to define
$q$-holonomic $\Trp$-modules.

For $\al=(\al_1,\dots,\al_r)\in \BZ^r$ let $| \al | =\sum_{j=1}^r \al_j$ and
$$ 
\sM^\al = \prod_{j=1}^r \sM_j^{\al_j}, \quad  
\sL^\al = \prod_{j=1}^r \sL_j^{\al_j}
\,.
$$
Consider the increasing filtration $\calF$ on $\cT_{r,+}$ given by
\be
\label{eq.FN}
\calF_k \cT_{r,+} = \{\text{$\cR$-span of all monomials
$\sM^\alpha \sL^\beta$ with $\alpha,\beta \in \BN^r$
and  $| \al |  + |  \beta | \le k$}\} \,.
\ee
Let $M$ be a finitely generated $\Trp$-module. 
The filtration $\calF$ on $\cT_{r,+}$
induces an increasing filtration on $M$, defined by
$\calF_k M=\calF_k \cT_{r,+} \cdot J$ where $J$ is a finite set of 
generators of $M$ as a $\cT_{r,+}$-module. It is easy to see that 
$\calF_k M$ is independent of $J$, and depends only on the $\cT_{r,+}$-module 
$M$. Note that $\calF_k \cT_{r,+}$, and consequently $\calF_k M$, are 
finitely generated $\Zt$-modules for all $k\in \BN$.
An analog of Hilbert's theorem for this non-commutative setting holds:
the $\cR$-dimension of $\calF_k M$
is a polynomial in $k$, for big enough $k$. The degree of this polynomial
is called the dimension of $M$, and is denoted by $d(M)$.

In \cite[Theorem 1.5.3]{Sab} Sabbah proved that $d(M) = 2r -\codim(M)$, where
\be
\label{eq.codim} 
\codim(M) = \min\{ j\in \BN \mid \Ext^j_{\Trp}(M, \Trp) \neq 0 \} \,.
\ee
He also proved that $d(M) \ge r$ if $M$ is  non-zero and does not have
monomial torsion. Here a monomial torsion is a monomial  $P$ in  $\Trp$
such that $Px=0$ for a certain non-zero $x\in M$. It is easy to see that 
$M$ embeds in the $\Tr$-module $\Tr\otimes_{\Trp} M$ if and only if $M$ 
has no monomial torsion. Of course if $M=0$ then $d(M)=0$.

\begin{definition}
\rm{(a)}
A $\Trp$-module $M$ is $q$-holonomic if $M$ is
finitely generated, does not have monomial torsion, and $d(M)\le r$.
\newline
\rm{(b)}
An element $f\in M$, where $M$ is a $\Trp$-module not necessarily 
finitely generated, is $q$-holonomic over $\Trp$ if $\Trp\cdot f$ is a 
$q$-holonomic $\Trp$-module.
\end{definition}

\subsection{The case of $\Tr$}
\label{sub.qholomany}

Let $M$ be a non-zero finitely generated left $\cT_r$-module.
Following \cite[Section 2.1]{Sab}, the codimension and dimension of $M$ are
defined in terms of homological algebra by an analog of \eqref{eq.codim}:
\be
\label{eq.codim2}
\codim(M) := \min\{ j\in \BN \mid \Ext^j_{\cT_r}(M, \cT_r) \neq 0 \},
\quad \dim(M) := 2r - \codim(M).
\ee

The key Bernstein inequality (proved by Sabbah \cite[Thm.2.1.1]{Sab}
in the $q$-case) asserts that if $M\neq 0$ is a finitely generated
$\Tr$-module, then $\dim(M) \ge r$. For $M=0$ let $\dim(M)=0$.

\begin{definition}
\rm{(a)}
A $\cT_r$-module $M$ is $q$-holonomic if $M$
is finitely generated and $\dim(M)\le r$.
\newline
\rm{(b)}
An element $f\in M$, where $M$ is a $\Tr$-module not necessarily 
finitely generated, is $q$-holonomic over $\Tr$ if $\Tr\cdot f$ 
is a $q$-holonomic $\Tr$-module.
\end{definition}

Thus a non-zero finitely generated $\Tr$-module is $q$-holonomic if and
only if it is minimal in the complexity measured by the dimension.

Next we compare $q$-holonomic modules over $\Tr$ versus over $\Trp$.
To do so, we use the following proposition of Sabbah~\cite[Cor.2.1.5]{Sab}.

\begin{proposition} 
\label{r.01}
Suppose $N$ is a $\Trp$-module and $M = \Tr \ot_{\Trp} N$.
\begin{itemize}
\item[(a)] 
If $N$ is $q$-holonomic over $\Trp$ then $M$ is $q$-holonomic over $\Tr$.
\item[(b)] 
Suppose $M$ is $q$-holonomic over $\Tr$ then there is a $\Trp$-submodule 
$N'\subset N$ such that $N'$ is $q$-holonomic over $\Trp$ and 
$M = \Tr \ot_{\Trp} N'$.
\end{itemize}
\end{proposition}
Actually, part (b) of the above proposition is contained in the proof 
of~\cite[Cor.2.1.5]{Sab}.

\begin{proposition}
\label{prop.WWr}
Suppose $f\in M$, where $M$ is a $\Tr$-module. Then $f$ is $q$-holonomic
over $\Tr$ if and only if it is $q$-holonomic over $\Trp$.
\end{proposition}

\begin{proof}
We can assume that $f\neq 0$ and that $M=\cT_r \cdot f$. 
Let $N= \Trp \cdot f\subset M$. Then $N$ is a
$\Trp$-submodule of $M$ without monomial torsion and $M =\Tr\otimes_{\Trp} N$. 
Proposition~\ref{r.01} implies that if $f$ is $q$-holonomic
over $\Trp$, then $f$ is $q$-holonomic over $\Tr$.

We now prove the converse. Assume that $M$  is $q$-holonomic over $\Tr$.
By Proposition~\ref{r.01}, there is a $\Trp$-submodule $N'\subset N$ such 
that $d(N')=r$ and $M= \Tr\otimes_{\Trp} N'$. Since $\Trp$ is Noetherian, 
we can assume $N'$ is $\Trp$-spanned by $p_1 f,\dots,p_k f$, where 
$p_i\in \Trp$.

{\bf Claim 1.} Suppose $a,b$ are elements of a $\Trp$-module and $a,b$ are
$q$-holonomic over $\Trp$. Then $a+b$ is $q$-holonomic over $\Trp$.
\newline
{\em Proof of claim 1.} Since $\cF_k(a+b) \subset \cF_k(a) + \cF_k(b)$, we have
$$
\rk_\Zt (\cF_k(a+b)) \le  \rk_\Zt (\cF_k(a)) +  \rk_\Zt (\cF_k(b)) =O(n^r) \,,
$$
which shows that $d(\Trp (a+b)) \le r$, and hence $a+b$ is $q$-holonomic
over $\Trp$.

{\bf Claim 2.} Suppose $a \in M$ is $q$-holonomic over $\Trp$, then $pa$
is $q$-holonomic over $\Trp$ for any $p\in \Tr$.
\newline
{\em Proof of claim 2.} Let $\fM_l\subset \Trp$ be the set of all monomials
$\sM^\al \sL^\beta$ with total degree $\le  l$. There is a monomial
$\mathfrak m$ such that $\fm p \in \Trp$. Choose $l$ such that
$\fm \in \fM_l$ and $\fm p \in \cF_l$. Then for all positive integers
$N$, $\cF_N p \subset \sum_{\fm \in \fM_l} \fm^{-1} \cF_{N+l}$. Hence
$$
\rk_\Zt (\cF_N(p a)) \le \sum_{\fm \in \fM_l} O((N+l)^r) = O(n^r) \,,
$$
which proves that $p a$ is $q$-holonomic over $\Trp$.

Let us return to the proof of the proposition. Since $d(N')=r$, each of
$p_1 f, \dots, p_k f$ is $q$-holonomic over $\Trp$. Because
$M= \Tr\otimes_{\Trp} N'$, there are $s_1,\dots,s_k \in \Tr$ such that
$f = \sum_{i=1}^k s_i p_i f$. Claims 1 and 2 show that $f$ is
$q$-holonomic over $\Trp$.
\end{proof}

\begin{remark}
When $r=1$, the above definition of $q$-holonomicity is equivalent
to the one given in Section~\ref{sub.Zextension}.
\end{remark}


\section{Properties of $q$-holonomic modules}
\label{sec.modules}

The class of $q$-holonomic $\Tr$-modules is closed under several natural
operations. We will collect these operations here, and refer to
Sabbah's paper for complete proofs. Below $q$-holonomic means 
$q$-holonomic over $\Tr$.

\subsection{Sub-quotients and extensions} 
\label{sub.subquotient}

By~\cite[Cor 2.1.6]{Sab}, we have the following.

\begin{proposition}
\label{r.subquot}
\rm{(a)} 
Submodules and quotient modules of $q$-holonomic modules are $q$-holonomic.
\newline
\rm{(b)}
Extensions of $q$-holonomic modules by $q$-holonomic modules
are $q$-holonomic.
\end{proposition}

\no{
\subsection{Restriction}
\label{sub.restriction}

Next we discuss the restriction of a $\Tr$-module $M$ with support in
$n_r-a=0$ to a $\cW_{r-1}$-module, following~\cite[Sec.1.4]{Sab}. By 
support in $n_r-a=0$, we mean that every element $m$ of $M$ is annihilated by
a product $\prod_{j \in \BZ}(n_r - q^j a)^{\gamma_j}$ where $\gamma_j \in \BN$ is
nonzero for finitely many $j$. When this is satisfied, we can
define $M_0=\cup_{p \in \BN} \mathrm{Ker}((n_1-a)^p)$. 
Sabbah in~\cite[Prop.1.4.1]{Sab} 
constructs a $\cW_{r-1}$-module structure on $M_0$ and shows that
\begin{equation} 
\label{eq.MM0}
M \simeq \cR[\sM_r^{\pm 1}] \otimes_{\cR} M_0
\end{equation}

The remark at the end of~\cite[Sec.2.1]{Sab} implies the following.

\begin{proposition}
\label{r.restriction}
With the above notation, $M$ is a $q$-holonomic $\Tr$-module if and only
if $M_0$ is a $q$-holonomic $\cW_{r-1}$-module.
\end{proposition}
}

\subsection{Push-forward} 
\label{sub.pushforward}

Recall that $\sM^\al= \sM_1^{\al_1} \dots \sM_r^{\al_r}$ for 
$\al=(\al_1,\dots,\al_r)\in \BZ^r$. Suppose $A$ is an $r\times s$ matrix
with integer entries. Let $\sM=(\sM_1,\dots,\sM_r)$ and
$\sM'=(\sM'_1,\dots,\sM'_s)$. There is an $\cR$-linear  map
$$
\cR[\sM^{\pm 1}] \to \cR[\sM'^{\pm 1}], \qquad
\sM^{\al} \mapsto (\sM')^{A^T \al}
$$
where $A^T$ is the transpose of $A$. If $M$ is a $\cT_r$-module, define
$$
(T_A)_*(M) = \cR[\sM'^{\pm 1}] \otimes_{\cR[\sM^{\pm 1}]} M \,,
$$
which is a $\cT_s$-module via the following action:
$$
\sM'_i (P \otimes m) = (\sM'_i P) \otimes m, \qquad
\sL'_i (P \otimes m) = \tau_i(P) \otimes \sL_1^{A_{1,i}} \dots \sL_r^{A_{r,i}} m 
\,,
$$
where $\tau_i: \cR[\sM'^{\pm 1}]\to \cR[\sM'^{\pm 1}]$ is the $\cR$-algebra 
map given by $\tau_i(\sM'_j)= q^{\delta_{i,j}} \sM'_j$.
In~\cite[Prop.2.3.3]{Sab} Sabbah proves:

\begin{proposition}
\label{r.pushforward}
If $M$ is a $q$-holonomic $\cT_r$-module and $A$ is an $s\times r$ matrix 
with integer entries, then $(T_A)_*(M)$ is a $q$-holonomic $\cT_s$-module.
\end{proposition}

\subsection{Symplectic automorphism}
\label{sub.symplectic}

Next we discuss a symplectic automorphism group of the quantum Weyl
algebra.
Suppose $A,B,C,D$ are $r\times r$ matrices with integer entries and 
$$
X=\begin{pmatrix}
A & B \\C & D
\end{pmatrix}.
$$
Define an $\cR$-linear map $\phi_X: \Tr \to \Tr$ by
$$ 
\phi_X(\sM^\al \sL^\beta)= \sM^{A \al} \sL^{B \beta} \sM^{C \al} \sL^{D \beta} \,.
$$
Then $\phi_X$ is an $\cR$-algebra automorphism if and only if $X$ is a 
symplectic matrix.

Suppose $X$ is symplectic and $M$ is a $\Tr$-module. Then $\phi_X$ induces 
another $\Tr$-module structure on $M$, where the new action 
$u \cdot_{\phi_X} x \in M$ for 
$u\in \Tr$ and $x\in M$ is defined by
$$ 
u \cdot_{\phi_X} x = \phi_X(u) \cdot x \,.
$$
This new $\Tr$-module is denoted by $(\phi_X)^*(M)$. 

\begin{proposition}
\label{prop.symplectic}
$M$ is $q$-holonomic if and only if $(\phi_X)^*(M)$ is $q$-holonomic.
\end{proposition}

\begin{proof}
This follows easily from the fact that the ext groups of $(\phi_X)^*(M)$
are isomorphic to those of $M$, and the fact that codimension and dimension
can be defined using the ext groups alone; see Equation~\eqref{eq.codim2}.
\end{proof}

In particular, when
\begin{equation}
\label{eq.Four}X= \begin{pmatrix}
0 & I  \\ -I  & 0
\end{pmatrix}
\end{equation}
then $(\phi_X)^*(M)$ is called the Mellin transform of $M$, and is 
denoted by $\fM(M)$, following~\cite[Sec.2.3]{Sab}. In particular,

\begin{corollary}
\label{cor.mellin}
If $M$ is a $q$-holonomic $\cT_r$-module, so is $\mathfrak{M}(M)$.
\end{corollary}
 
\no{
 For $1\le i \le r$ let $A(i)$ be the $i$-th column of $A$. 
Define
$$
M'= M^{A} L^B, L'= M^C L^D.
$$
Then the map $\phi(M,L)=(M',L')$ is an algebra automorphism of $\Tr$
if and only if the matrix
$$
X=\begin{pmatrix}
A & B \\C & D
\end{pmatrix}
$$
is symplectic. We will call such $\phi$ symplectic automorphisms
of $\Tr$. Given such an automorphism $\phi$, and a $cT_r$-module
$M$, let $\phi(M)$ denote the set $M$ with $cT_r$-module action
$(P,f) \in cT_r \times M \mapsto \phi(P)f$. In \cite[Sec.1.3]{Sab}
Sabbah proves:

If $M$ is a $q$-holonomic $\cT_r$-module and $\phi$ is a symplectic
automorphism of $\cT_r$, then $\phi(M)$ is $q$-holonomic too.
}

Another interesting case is when
$$
X=\begin{pmatrix}
A & 0 \\0 & (A^T)^{-1}
\end{pmatrix}
$$
where $A \in \GL(r,\BZ)$. 

\subsection{Tensor product}
\label{sub.tensor}

Suppose $M,M'$ are $\Tr$-modules.
One defines their box product $M \boxtimes M'$ and their convoluted 
box product $M \hbt M'$, which are  $\Tr$-modules, as follows. As an 
$\cR[\sM_1^{\pm1}, \dots, \sM_r^{\pm1}]$-module,
$$
M \boxtimes M'=M\ot_{\cR[\sM_1^{\pm1}, \dots, \sM_r^{\pm1}]} M'\,,
$$
and the $\Tr$-module structure is given by
\be \label{eq.1a1}
\sM_i(x \otimes x')= \sM_i(x) \otimes x' = x \otimes \sM_i(x'), \qquad
\sL_i(x \otimes x')= \sL_i(x) \otimes \sL_i(x') \,.
\ee

Similarly, as $\cR[\sL_1^{\pm1}, \dots, \sL_r^{\pm1}]$-module,
$$
M \hbt M'=M\ot_{\cR[\sL_1^{\pm1}, \dots, \sL_r^{\pm1}]} M'\,,
$$
and the $\Tr$-module structure is given by
\be
\label{eq.1b1}
\sL_i(x \otimes x')= \sL_i(x) \otimes x' = x \otimes \sL_i(x'), \qquad
\sM_i(x \otimes x')= \sM_i(x) \otimes \sM_i(x') \,.
\ee

\begin{proposition} 
\label{r.tensor}
Suppose $M,M'$ are $q$-holonomic $\Tr$-modules. Then both 
$M \boxtimes M'$ and $M \hbt M'$ are $q$-holonomic.
\end{proposition}

\begin{proof}
The case of $M \boxtimes M'$ is a special case 
of~\cite[Proposition 2.4.1]{Sab}, while the case of $M \hbt M'$ follows 
from the case of the box  product via the Mellin transform, since 
$$
M \hbt M'= \fM^{-1}(\fM(M) \boxtimes \fM(M')) \,.
$$
\end{proof}

\subsection{$q$-holonomic modules are cyclic}
\label{sub.cyclic}

An interesting property of $q$-holonomic $\Tr$-modules $M$ is that they are
cyclic, i.e., are isomorphic to $\Tr/I$ for some left ideal $I$ of $\Tr$.
This is proven in \cite[Cor.2.1.6]{Sab}.


\section{Properties of $q$-holonomic functions}
\label{sec.properties}

\subsection{Fourier transform}
\label{sub.fourier}
\def\fF{\mathfrak F}
The idea of the Fourier transform $\fF(f)$ of a function $f \in S_r(V)$
is the following: the Fourier transform is simply the generating series
\be
\label{eq.Ff}
(\fF f)(z) = \sum_{n \in \BZ^r} f(n) z^n,
\ee
where $n=(n_1,\dots, n_r)$ and $z^n = \prod_{j=1}^r z_j ^{n_j}$. More formally,
let $\check{S}_r(V)$ denote the set whose elements are the expressions
of the right hand side of~\eqref{eq.Ff}. Then $\check{S}_r(V)$ is an 
$\cR$-module equipped with an action of $\Tr$ defined by
$$
(\sM_i g)(z)= g(z_1,\dots,z_{i-1},qz_i,z_{i+1},\dots,z_r),
\qquad (\sL_i g)(z)= z_i^{-1} g(z)
$$
for $g(z) \in \check{S}_r(V)$. The $\Tr$-module structure on $S_r(V)$
and $\check{S}_r(V)$ is chosen so that the following holds.

\begin{lemma}
\label{lem.F}
\rm{(a)}
The map $\fF: S_r(V) \to \check{S}_r(V)$ given by Equation~\eqref{eq.Ff}
is an isomorphism of $\Tr$-modules.
\newline
\rm{(b)}
$f\in S_r(V)$ is $q$-holonomic if and only if its Fourier
transform $\fF(f)$ is.
\newline
\rm{(c)} The relation of the Fourier and Mellin transform are as follows.
If $f \in S_r(V)$, then
$$
M_{\fF(f)} = \mathfrak{M}(M_f) \,.
$$
\end{lemma}

Now, suppose that $V$ is a commutative $\Zt$-algebra.
Then $S_r(V)$ is a $\Tr$-algebra. Hence $\check{S}_r(V)$, via $\fF$, 
inherits a product, known as the Hadamard product $\circledast$, given by
$$
f(z)=\sum_{n \in \BZ^r} f(n) z^n, \qquad
g(z)=\sum_{n \in \BZ^r} g(n) z^n, \qquad
f(z) \circledast g(z)=\sum_{n \in \BZ^r} f(n) g(n) z^n
$$
Of course $\fF$ is an isomorphism of algebras. Note that the action of 
$\Tr$ on the product of two functions are given by
\be
\label{eq.prod1a}
\sM_j (fg) = \sM_j(f) g = f \sM_j(g), \quad \sL_j(fg)= \sL_j(f) \sL_j(g).
\ee

The $\Zt$-subspace of $\check{S}_r(V)$ consisting of all power series 
with finite support is isomorphic to the group ring $\Zt[\BZ^r]$ and has
a natural product defined by multiplication of power series
in $z$, that corresponds to a convolution product on the subset of $S_r(V)$ 
consisting of functions with finite support. Unfortunately, multiplication 
of power series in $z$ cannot be extended to the whole $\check{S}_r(V)$. 
However, this convolution product can be extended to bigger subspaces 
as follows. For an integer  $k$ with $0 \leq k \leq r$, let 
$S_{r,k}(V)$ denote the set of
functions $f: \BZ^r=\BZ^k \times \BZ^{r-k} \to V$ such that for each
$n \in \BZ^k$, the support of $f(n,\cdot): \BZ^{r-k} \to V$ is a finite 
subset of $\BZ^{r-k}$. Let $S_{r,k}^\str(V)$ denote the set of
functions $f: \BZ^r=\BZ^k \times \BZ^{r-k} \to V$ that vanish outside
$J \times \BZ^{r-k}$ for some finite subset $J \subset \BZ^k$ ($J$
in general depends on $f$).

For $f\in S_{r,k}(V)$ and $g \in S_{r,k}^\str(V)$ one can define the 
convolution $f*g\in S_r(V)$ by
$$  
(f*g)(n)=\sum_{m \in \BZ^r} g(m) f(n-m) \,.
$$
The right hand side is well-defined since there are only a finite number 
of non-zero terms. The convolution is transformed into the product of 
power series by the Fourier transform: 
for  $f\in S_{r,k}^\str(V)$ and $g \in S_{r,k}^\str(V)$ we have:
\be
\label{eq.prod1}
\cF(f*g)= \cF(f) \cF(g) \,.
\ee

Note that
\be
\label{eq.prod1b}
\sM_j (f*g) = \sM_j(f) * \sM_j (g), \quad 
\sL_j(f*g)= \sL_j(f)*g = f* \sL_j(g) \,.
\ee

\subsection{Closure properties}
\label{sub.properties}

In this section we summarize the closure properties of the class of 
$q$-holonomic functions. These closure properties were known in the classical 
case (non $q$-case, see \cite{Z90}) and we are treating the $q$-case.
Theorem \ref{r.sum} below and Theorem \ref{thm.multisum} in the next section 
were known as folklore, but to the best of our knowledge, there were no 
proofs given in the literature. The main goal of this survey is to give proofs 
to these fundamental results. 

\begin{theorem}
\label{r.sum} 
The class of $q$-holonomic functions are closed under the following 
operations.
\begin{itemize}
\item[(a)] 
Addition: Suppose $f,g\in S_r(V)$ are $q$-holonomic. 
Then $f+g$ is $q$-holonomic.
\item[(b)] 
Multiplication: Suppose $f,g\in S_r(V)$ are $q$-holonomic. 
Then $fg$ is $q$-holonomic.
\item[(b')] 
Convolution: 
Suppose $f \in S_{r,k}(V)$ and $g \in S_{r,k}^\str(V)$
are $q$-holonomic. Then $f \! \ast  g $  is $q$-holonomic.
\item[(c)] 
Affine substitution: 
Suppose $f\in S_r(V)$ is $q$-holonomic, $A$ is an $r\times s$ 
matrix with integer entries, and $b \in \BZ^s$. Then $g\in S_s(V)$ defined by
$$
g(n) :=f(An + b)
$$ 
is $q$-holonomic.
\item[(d)]
Restriction: Suppose $f\in S_r(V)$ is $q$-holonomic and $a\in \BZ$. Then
$g\in S_{r-1}(V)$ defined by
$$
g(n_1,\dots,n_{r-1})= f(n_1,\dots,n_{r-1},a)
$$
is $q$-holonomic.
\item[(e)]
Extension: Suppose $f\in S_r(V)$ is $q$-holonomic.
Then $h \in S_{r+1}(V)$ defined by
$$
h(n_1,\dots,n_{r+1})= f(n_1,\dots,n_{r})
$$
is $q$-holonomic.
\item[(f)] 
Rescaling $q$: Suppose $f\in S_r(V)$ is $q$-holonomic where 
$V=\bk(q)\otimes_\bk V_0$ is a $\bk(q)$-vector space. Fix a nonzero integer 
$c$ and let $\sigma: \bk(q)\to\bk(q^c)$ be the field isomorphism given by 
$\sigma(q)=q^c$, $W=\bk(q^c)\otimes_\bk V_0$ and 
$g=\sigma \circ f \in S_r(W)$.
If $f$ is $q$-holonomic, then $g$ is $q$-holonomic.
\end{itemize}
\end{theorem}

\begin{proof}
(a) 
Recall that  $M_f = \Tr f$, which is a $\Tr$-module.
The map
$$
M_f \oplus M_g  \to S_r(V) \,,
$$
given by $x\oplus y \mapsto x+y$ is $\Tr$-linear and its image contains
$M_{f+g}$. Thus, $M_{f+g}$ is a subquotient of $M_f \oplus M_g$.
By Proposition \ref{r.subquot}, $M_{f+g}$ is
$q$-holonomic.

(b) From \eqref{eq.1a1} and \eqref{eq.prod1a} one sees that  the map
$$
M_f  \boxtimes M_g  \to S_r(V) \,,
$$
given by $x\otimes y \mapsto xy$ is $\Tr$-linear and its image contains
$M_{fg}$. Thus, $M_{fg}$ is a subquotient of
$M_f \boxtimes M_g$. By Proposition \ref{r.subquot}, $M_{fg}$ is
$q$-holonomic.

(b') From \eqref{eq.1b1} and \eqref{eq.prod1b} one sees that  the map
$$
M_f  \hbt M_g  \to S_r(V) \,,
$$
given by $x\otimes y \mapsto x*y$ is $\Tr$-linear and its image contains
$M_{f*g}$.  By Proposition \ref{r.subquot}, $M_{f*g}$ is
$q$-holonomic.

(c) For $b\in \BZ^s$ and $f\in S_s(V)$ let $h':= L^b h$. We have 
$h'(n) = h(n+b)$. Then 
$$
M_{h'}=\cT_s h'  = ( \cT_s  \sL^b) h= \cT_s h= M_h \,.
$$
Thus, $h$ is $q$-holonomic if and only if $h'$ is. Hence, we can assume 
that $b=0$ in proving (c).

Consider the linear map $\BZ^s \to \BZ^r$, $n=(n_1,\dots,n_s) \to 
An=n'= (n_1',\dots,n'_r)$. Then $g(n)=f(An)$. Observe that 
$q^{n'_i}=q^{A_{i1}n_1+ \dots A_{is}n_s}$, i.e.,
$\sM'_i=\sM_1^{A_{i1}} \dots \sM_s^{A_{is}}$. Moreover, 
$$
(\sL^\beta g)(n) = g(n+\beta)= f(An + A \beta)= ((\sL')^{A\beta} f)(An) \,.
$$
It follows that the $\cR$-linear map $\psi: (T_A)_*(S_r(V)) \to S_s(V)$, 
where $(T_A)_*(S_r(V))$ is the push-forward of $S_r(V)$ 
(see Section \ref{sub.pushforward}), given by
\be 
\psi(\sM^\al \ot h) =  \sM^\al (h \circ A)
\ee
is a $\cT_s$-module homomorphism.  Since
\be 
\sM^\al \sL^\beta g = \psi(M^\al \ot (\sL')^{A\beta} f),
\ee
and the set of all $\sM^\al \sL^\beta g$ $\cR$-spans $\cT_s \cdot g$, 
it follows that $\cT_s g$ is a submodule of $\psi(\cT_r f)$.
Since $f$ is $q$-holonomic, Propositions~\ref{r.pushforward} 
and~\ref{r.subquot} imply that $\cT_s g$ is a submodule of the $q$-holonomic 
module $\psi(\cT_r f)$, hence is $q$-holonomic. 

(d) and (e) are a special cases of (c).

(f) Observe that
$$
\sigma (a(q) \sL^\alpha \sM^\beta f)= \sigma (a(q) \sL^\alpha \sM^\beta)
g
$$
where
$$
\sigma (a(q) \sL^\alpha \sM^\beta) = a(q^c) \sL^\alpha \sM^{c \beta} \,.
$$
Assume that $f$ is $q$-holonomic with respect to $\Tr$. Using Proposition
\ref{prop.WWr}, it follows that that $f$ is $q$-holonomic with respect to 
$\Trp$, and it suffices to show that $g$ is $q$-holonomic with respect to 
$\Trp$.
Recall the good filtration $\calF_k$ on $\Trp f$ from Equation ~\eqref{eq.FN}.
It follows that $\calF_k g$ is the span of $\sL^\alpha \sM^\beta \sigma f$ 
for $|\alpha|+|\beta| \leq k$. Equivalently, it is the span of 
$\sM^\gamma \sL^\alpha \sM^{c \beta} \sigma f$ where $\gamma_i \in \{0,1,\dots,
|c|-1\}$ for $i=1,\dots,r$ and $|\gamma|+|\alpha|+c|\beta| \leq k$.
Since $\sM^\gamma \sL^\alpha \sM^{c \beta} \sigma f$ = $\sM^\gamma \sigma(
\sL^\alpha \sM^{\beta}   f)$, and $|\alpha|+|\beta| \leq k$ and the number
of $\gamma$ is $O(1)$, and the dimension of the extension $\bk(q)/\bk(q^c)$ 
is finite, it follows that
the dimension of the span of $\calF_k g$ is at most the dimension of
the span of $\calF_k f$, times a constant which is independent of $k$. 
Hence, $g$ is $q$-holonomic with respect to $\Trp$.  
\end{proof}

\subsection{Multisums}
\label{sub.multisums}

In this section we prove that multisums of $q$-holonomic functions are
$q$-holonomic. This important closure property of $q$-holonomic functions
(even in the case of multisums of $q$-proper hypergeometric functions)
is not proven in the literature, since the paper of Wilf-Zeilberger~\cite{WZ} 
predated Sabbah's paper~\cite{Sab} that provided the definition of 
$q$-holonomic functions. On the other hand, quantum knot invariants 
(such as the colored Jones and the colored HOMLFY polynomials) are 
multisums of $q$-proper hypergeometric functions~\cite{GL,GLL}, and hence 
$q$-holonomic. It is understood that a modification of the proof in the 
classical (i.e., $q=1$) case ought to work in the $q$-holonomic case.
At any rate, we give a detailed proof, which was a main motivation to 
write this survey article on $q$-holonomic functions.

Recall that $S_{r,1}(V)$ the set of all functions
$f: \BZ^r \to V$ such that for every $(n_1,\dots,n_{r-1})\in \BZ^{r-1}$,
$f(n_1,\dots,n_r)=0$ for all but a finite number of $n_r$.

\begin{theorem}
\label{thm.multisum}
\rm{(a)} Suppose $f\in S_{r,1}(V)$ is $q$-holonomic. Then,
$g\in S_{r-1}(V)$, defined by
$$
g(n_1,\dots,n_{r-1})= \sum_{n_r \in \BZ} f(n_1,\dots,n_r) \,,
$$
is $q$-holonomic.
\newline
\rm{(b)} Suppose $f\in S_r(V)$ is $q$-holonomic. Then
$h\in S_{r+1}(V)$ defined by
\be
\label{eq.multi}
h(n_1,\dots,n_{r-1},a,b)=\sum_{n_r=a}^b f(n_1,n_2,\dots,n_r)
\ee
is $q$-holonomic.
\end{theorem}

\begin{proof}
(a) Let $\nu \in S^\str_{r,1}(V)$ be defined by
$$\nu(n_1,\dots,n_r)=\delta_{n_1,0} \dots \delta_{n_{r-1},0}.$$
Lemma \ref{r.delta} and Theorem \ref{r.sum} show that $\nu$ is $q$-holonomic.
Hence $g'= f\!\ast \nu$ is $q$-holonomic. Note that $g'$ is constant on the 
last variable, and
$$ 
g'(n_1,\dots,n_r)= g(n_1,\dots,n_{r-1}) \,.
$$
In particular, 
$ g(n_1,\dots,n_{r-1})= g(n_1,\dots,n_{r-1},0).$
By Theorem \ref{r.sum}(d), $g$ is $q$-holonomic.

\no{

Since $f$ is $q$-holonomic, 

$\fF(f)\in \check{S}_{r-1,1}(V)$
is $q$-holonomic by Lemma \ref{lem.F}. The element
$\mu\in \check{S}_{r-1,1,s}(V)$ defined by
$$
\mu = \sum_{n\in \BZ} z_r^n
$$
is $q$-holonomic. Indeed, $\mu=\fF(\nu)$ where $\nu \in S_{r-1,1,s}(V)$
satisfies 

$\nu(n_1,\dots,n_r)=\delta_{n_1,0} \dots \delta_{n_{r-1},0}$.

By part (b) of Theorem~\ref{r.sum},
$\fF(f) \mu \in \check{S}_r(V)$ is $q$-holonomic.
By Lemma \ref{lem.F}, $G = \fF^{-1}(\fF \mu)$ is $q$-holonomic.
Note that $G$ is constant on the last variable, and
$$
G(n_1,\dots,n_{r-1}, n_r)= g(n_1,\dots,n_{r-1}).
$$
By part (c) of Theorem~\ref{r.sum}, it follows that $g$ is $q$-holonomic.
}

(b) follows from (a) using the identity
$$
h(n_1,\dots,n_{r-1},a,b) = g(n_1,\dots,n_r) H(n_r-a) H(b-n_r)
$$
where $H(n)$ is the Heaviside function~\eqref{eq.heavyside}.
\end{proof}

\subsection{Extending from $\BN^r$ to $\BZ^r$}
\label{sub.extending}
Here is an extension of Lemma \ref{lem.qholo1} to several variables.
\begin{proposition}
\label{prop.NZ}
\rm{(a)}
If $f \in S_r(V)$ is $q$-holonomic and $g \in S_{r,+}(V)$
is its restriction to $\BN^r$, then $g$ is $q$-holonomic.
\newline
\rm{(b)}
Conversely, if $g \in S_{r,+}(V)$ is $q$-holonomic and
$f \in S_r(V)$ is the extension of $g$ to $\BZ^r$ by zero (i.e.,
$f(n)=g(n)$ for $n \in \BN^r$, $f(n)=0$ otherwise), then $g$ is
$q$-holonomic.
\end{proposition}

\begin{proof}
(a) For $h \in S_r(V)$, let $\mathrm{Res}(h) \in S_{r,+}(V)$ denote the
restriction of $h$ to $\BN^r \subset \BZ^r$. If $P \in \Trp$, observe that
$\mathrm{Res}(Pf)=Pg$, and consequently, $\mathrm{Res}(\calF_k f)=\calF_k g$.
It follows that if $f \in S_r(V)$ is  $q$-holonomic and
$g=\mathrm{Res}(f)$, then $g$ is $q$-holonomic.

(b) Let $I=\Ann(g)\subset \Trp$ and $\tI=\Tr I$ be its extension in $\Tr$. 
We have the following short exact sequence of $\Tr$-modules
\be
0 \to \tI \cdot f \to \Tr \cdot f \to (\Tr \cdot f) /(\tI \cdot f) \to 0 \,.
\ee

We claim that:
\begin{itemize}
\item[(1)]
$(\Tr \cdot f)/(\tI \cdot f)$ is $q$-holonomic over $\Tr$. 
\item[(2)]
$\tI \cdot f$ is $q$-holonomic over $\Tr$.
\end{itemize}
If that holds, Proposition~\ref{r.subquot} concludes the proof.

To prove (1), note that  $(\Tr \cdot f)/(\tI \cdot f)$ is a quotient of 
$\Tr/\tI= \Tr \otimes_{\Trp}  (\Trp/I)$. By Propositions \ref{r.01} 
and \ref{r.subquot}, $(\Tr \cdot f)/(\tI \cdot f)$ is 
$q$-holonomic over $\Tr$.

To prove (2), suppose $I$ is generated by $p_1,\dots,p_k$. It suffices to 
prove that each $p_j f$ is $q$-holonomic over $\Tr$. We prove this
by induction on $r$. For $r=1$, it is clear. Suppose it holds for $r-1$.
There is a finite set $J\subset \BZ$ such that the support of $p_j f$ is 
in $\cup_{0 \leq k \leq r-1} (\BZ^k \times J \times \BZ^{r-1-k})$. Without loss
of generality we can assume that $J$ consists of one element. In that case,
the induction hypothesis concludes that $p_j f$ is $q$-holonomic.
\end{proof}

\begin{corollary}
\label{sub.ZN}
Theorems~\ref{r.sum} (where in part (c) we assume $A: \BN^s \to \BN^r$
and $b \in \BN^s$) and~\ref{thm.multisum} hold for $q$-holonomic 
functions over $\Trp$. 
\end{corollary}

\subsection{Modifying and patching $q$-holonomic functions}
\label{sub.alter}

In this section we discuss how a modification of a $q$-holonomic function
by another one is $q$-holonomic, and that the patching of $q$-holonomic 
functions on orthants is a $q$-holonomic function.

\begin{proposition}
\label{prop.modify}
Suppose $V$ is an $\cR$-vector space and $f\in S_r(V), g\in S_{r-1}(V)$ are 
$q$-holonomic.
\begin{itemize}
\item[(a)] 
If $f'\in S_r(V)$ differs from $f$ on a finite set, then $f'$ is 
$q$-holonomic.
\item[(b)] 
Suppose $a\in \BZ$. If $f'=f$ except on the hyperplane 
$\BZ^{r-1}\times \{a \}$, where $f'(n,a)= g(n)$, then 
$f'$ is $q$-holonomic.
\end{itemize}
Similar statements holds for the case when the domains of $f, g$ are 
respectively $\BN^r, \BN^{r-1}$.
\end{proposition}

\begin{proof}
(a) In this case, $f-f'$ is a finite linear combination of delta 
functions, which is $q$-holonomic by Theorem \ref{r.sum} and the 
$q$-holonomicity of the one-variable delta function. By 
Theorem \ref{r.sum}, $f'$ is $q$-holonomic.

(b) The function $\tg\in S_r(V)$, defined by 
$\tg(n_1,\dots,n_r)= g(n_1,\dots, n_{r-1})$ is $q$-holonomic by 
Theorem \ref{r.sum}.  We have
$$ 
f'= (1- \delta(n_r-a)) f + \delta(n_r-a) \tg \,.
$$
By Theorem \ref{r.sum}, $f'$ is $q$-holonomic. 
\end{proof}

Let $\BN_+=\BN$ and $\BN_-= \{ -n \mid n \in \BN\} \subset \BZ$. There is 
a canonical isomorphism $\BN_-\to \BN_+$ given by $n \mapsto -n$. We have 
$\BZ= \BN_+\cup \BN_-$.

For $\ve= (\ve_1,\dots,\ve_r) \in \{+, -\}^r$ define the $\ve$-orthant of 
$\BZ^r$ by
$$ 
\BN_\ve= \BN_{\ve_1} \times \BN_{\ve_2} \times \dots \times \BN_{\ve_r} 
\subset \BZ^r \,.
$$
The canonical isomorphism $\BN_-\to \BN$ induces a canonical isomorphism 
$\BN_\ve \cong \BN^r$, and a function $f: \BN_\ve \to V$ is called 
$q$-holonomic if its pull-back as a function on $\BN^r$ is $q$-holonomic.

\begin{proposition} 
\label{prop.patching}
A function $f\in S_r(V)$ is $q$-holonomic if and only if its restriction 
on each orthant is $q$-holonomic.
\end{proposition}

\begin{proof} 
If $f \in S_r(V)$ is $q$-holonomic, then its restriction to an orthant
is the restriction to $\BN^r$ of $A \circ f$ where $A \in \GL(r,\BZ)$
is a linear transformation. Part (c) of Theorem \ref{r.sum} together with
Proposition \ref{prop.NZ} conclude that the restriction of $f$ to each
orthant is $q$-holonomic.

Conversely, consider a function $f$ and its restriction $f_\epsilon$ to
the orthant $\BN_\epsilon$. Proposition \ref{prop.NZ} implies that
the extension $g_\epsilon$ of $f_\epsilon$ by zero is $q$-holonomic for all
$\epsilon$. Moreover, $f-\sum_\epsilon g_\epsilon$ is a function supported on
a finite union of coordinate hyperplanes. By induction on $r$, (the case
$r=1$ follows from Proposition \ref{prop.modify}) we may
assume that this function is $q$-holonomic. Part (b) of 
Proposition \ref{prop.modify} concludes the proof.
\end{proof}


\section{Examples of $q$-holonomic functions} 
\label{sub.basic}

Besides the $q$-holonomic functions of one variable given in 
Example~\ref{ex.fg1} (with domain extended to $\BZ$ via 
Lemma~\ref{lem.qholo1}), we give here some basic examples of $q$-holonomic 
functions. These examples can be used as building blocks in the 
assembly of more $q$-holonomic functions using the closure properties
of Section~\ref{sub.properties}.

Recall that for $n \in \BN$,
\be
\label{eq.qp}
(x;q)_n= \prod_{j=1}^n( 1- x q^{j-1}).
\ee

\begin{lemma} 
\label{r.delta}
The delta function $\BZ^2 \to \BQ(q)$, given by 
$(n,k)\to  \delta_{n,k}$, is $q$-holonomic.
\end{lemma}

\begin{proof} 
We have  $\delta_{n,k}= \delta(n-k)$. By Example \ref{ex.fg2} and 
Theorem \ref{r.sum}, $\delta_{n,k}$ is $q$-holonomic.
\end{proof}

For $n,k\in \BZ$, let
$$
F(n,k) = \begin{cases}
(q^n;q^{-1})_k,  & \text{if $k\ge 0$}\\
0 & \text{if $k<0$} \end{cases}
$$
$$
G(n,k) 
= \frac{F(n,k)}{(q^k;q^{-1})_k} = \begin{cases}
\frac{(q^n;q^{-1})_k}{(q^k;q^{-1})_k}  & \text{if $k\ge 0$}\\
0 & \text{if $k<0$} \end{cases}.
$$
Note that 
$$
G(n,k) = \binom nk _q = \frac{(q;q)_n}{(q;q)_k (q;q)_{n-k}}
$$ 
is the $q$-binomial coefficient~\cite{Lu} if 
$n \geq k \ge 0$.
In quantum topology (related to the colored HOMFLYPT polynomial~\cite{GLL}) 
we will also use the following extended $q$-binomial
defined for $n,k \in \BZ$ by

\be
\label{eq.qbn1}
H(n,k) = { x;  n \brack k} =\begin{cases} 0 \quad & \text{if } \ k < 0 \\
\prod_{j=1}^k
\frac{  x q^{ n -j+1}- x^{-1} q^{-n+j-1}  }{q^j - q^{-j}} &
\text{if } \ k \ge  0.
\end{cases}
\ee
Let
\be
\label{eq.K}
K(n,k,\ell)= { q^\ell;  n \brack k}
\ee

\begin{lemma}
\label{lem.FGH}
\rm{(a)} Suppose $\bk=\BQ$. Then, the functions $F$ and $G$ are 
$q$-holonomic.
\newline
\rm{(b)} Suppose $\bk= \BQ(x)$. Then, the function $H$ is $q$-holonomic.
\newline
\rm{(c)} The function $K$ is $q$-holonomic.
\end{lemma}

\begin{proof}
(a) There are 4 orthants (i.e. quadrants) of $\BZ^2$: 
$\BN_{+,+}, \BN_{-,+}, \BN_{+,-}, \BN_{-,-}$. On the last two quadrants, 
$F=0$ and hence are $q$-holonomic.

On the quadrant $\BN_{+,+}$ (corresponding to $n, k\ge 0$), $F(n,k)$ 
is the product of 2 functions
$$ 
F(n,k) =  (q^n;q^{-1})_n \times \frac 1 { (q^{n-k}, q^{-1})_{n-k}} \,.
$$
Both factors, considered as a function on $\BZ^2$, are $q$-holonomic 
by Example \ref{ex.fg1} (with extension to $\BZ$ by Lemma \ref{lem.qholo1}) 
and Theorem \ref{r.sum}. Hence, by Theorem \ref{r.sum} and 
Proposition \ref{prop.NZ}, $F(n,k)$ is $q$-holonomic on $\BN_{+,+}$.

Let us consider the quadrant $\BN_{-,+}$. Denote $m= -n$. Then 
$(m,k)\in \BN^2$, and
$$ 
F(n,k) = (-1)^k q^{-km} q^{-k(k-1)/2} (q^{m+k-1};q^{-1})_{m+k-1} 
\frac 1 {  (q^{k};q^{-1})_{k} } \,.
$$
All factors, considered as a function on $\BZ^2$, are $q$-holonomic 
by Example \ref{ex.fg1} (with extension to $\BZ$ by Lemma \ref{lem.qholo1}) 
and Theorem \ref{r.sum}. Hence, by Theorem \ref{r.sum} and 
Proposition \ref{prop.NZ}, $F(n,k)$ is $q$-holonomic on $\BN_{-,+}$.

Proposition \ref{prop.patching} shows that $F(n,k)$ is $q$-holonomic 
on $\BZ^2$. 

Since 
$$ 
G(n,k) = F(n,k) \times \frac 1{(q^k;q^{-1})_k}\,,
$$
where the second factor, considered as a function on $\BZ$, is 
$q$-holonomic, $G(n,k)$ is $q$-holonomic.

(b) 
For the $q$-hypergeometric function $H$, we can give a proof as in the 
case of $F$ and $G$. Alternatively, we can also deduce it using the closure
properties of $q$-holonomic functions as follows. We have
\begin{align*}
\prod_{j=1}^k   (x q^{n -j+1}- x^{-1} q^{-n+j-1})
= (-1)^k q^{-kn + \binom k2} x^{ -k} (x^2 q^{2(n-k+1) }; q^2)_k
\end{align*}
Using the Gauss binomial formula~\cite[Chpt.5]{KacCheung},
$$
( x;q^2)_k= \sum_{j=0} ^k (-1)^j 
q^{-k j-j} \binom{k}{j}_{q^2} x^j,
$$
we have
\be
\label{eq.HKidentity}
{ x; n \brack k} = 
(-1)^k (q-q^{-1})^k q^{k(k-1)} \frac{1}{(q;q)_k}
\sum_{j=0} ^k (-1)^j 
q^{2nj-3kj+j } \binom{k}{j}_{q^2} 
x^{2j}
\ee
The right hand side is a terminating sum of known $q$-holonomic functions.
Hence the extended $q$-binomial coefficient is $q$-holonomic.

(c) Let $x=q^\ell$ in Equation \eqref{eq.HKidentity}. The right hand side is 
a terminating sum of known $q$-holonomic functions, hence $K$ is $q$-holonomic.
\end{proof}

\begin{remark}
\label{rem.hilbertdim}
The above proof uses the closure properties of the class of $q$-holonomic 
functions. It is possible 
to give a proof using the very definition of $q$-holonomic functions via 
the Hilbert dimension.

\end{remark}


\section{Finiteness properties of $q$-holonomic functions}
\label{sec.Dfinite}

In this section we discuss finiteness properties of $q$-holonomic functions.

For any subset $\cL\subset \{\sL_1, \dots, \sL_r,\sM_1,\dots, \sM_r\}$ let 
$\TrL$ be the $\cR$-subalgebra of $\Tr$ generated by elements in $\cL$. For 
$i=1, \dots,r$ let $\cL_i=\{\sL_i,\sM_1,\dots, \sM_r\}$. 
Any non-zero element $P\in \TrLi$ has the form 
$$
P= \sum_{j=0}^k ( L_i)^j\,  a_j\,,
$$
where $a_j \in  \cR[\sM]:= \cR[\sM_1,\dots,\sM_r]$ and $a_k\neq 0$. We call 
$k$ the $L_i$-degree of $P$ and $a_k$ the {\em $L_i$-leading coefficient} 
of $P$.  

Consider the following finiteness properties for a function $f \in S_r(V)$.

\begin{definition} 
\label{def.finiteness}
Suppose $f\in S_r(V)$.
\newline
\rm{(a)}
We say that $f$ is strongly $\Tr$-finite (or that $f$ satisfies 
the elimination property) if 
for every subset $\cL$ of $\{\sM_1,\dots,\sM_r,\sL_1,\dots, \sL_r\}$ with
$r+1$ elements, $\Ann(f)\cap \TrL \neq \{0\}$.
\newline
\rm{(b)}
We say that $f$ is $\Tr$-finite if $\Ann(f)\cap \TrLi \neq \{0\}$ for every 
$i=1,\dots,r$.
\newline
\rm{(c)} We say that $f$ is integrally $\Tr$-finite if $\Ann(f)\cap \TrLi$ 
contains a non-zero element whose $L_i$-leading coefficient is 1, for every 
$i=1,\dots,r$.
\end{definition}

Our notion of $\Tr$-finiteness differs from the $\partial$-finiteness in
the Ore algebra $\BQ(q,\sM)\la \sL \ra$ considered in Koutschan's 
thesis~\cite[Sec.2]{Koutschan}. In particular, the Dirac $\delta$-function
$\delta_{n,0}$ is $q$-holonomic and $\cT_1$-finite (as follows from 
Theorem~\ref{thm.Dfinite} below) but not 
$\partial$-finite~\cite[Sec.2.4]{Koutschan}. 

The following summarizes the relations among the above flavors of finiteness.

\begin{theorem} 
\label{thm.Dfinite} 
Suppose $f\in S_r(V)$. One has the following implications among properties of 
$f$:
$$ 
\text{integrally $\Tr$-finite} \Rightarrow \text{$q$-holonomic}  
\Rightarrow \text{ strongly $\Tr$-finite} \Rightarrow \text{ $\Tr$-finite} \,.
$$
In other words,
\begin{itemize}
\item[(a)] If $f$ is integrally $\Tr$-finite, then $f$ is $q$-holonomic.
\item[(b)] If $f$ is $q$-holonomic, then $f$ is strongly $\Tr$-finite.
\item[(c)] If $f$ is strongly $\Tr$-finite, then $f$ is $\Tr$-finite.
\end{itemize}
\end{theorem}

\begin{proof} (c) is clear. 

For (a), 
suppose for each $i=1,\dots, r$ there is a non-zero $p_i\in \Ann(f)\cap \TrLi$ 
with $L_i$-leading coefficient  1 and $L_i$-degree $k_i$. 
Assume  $p_i = \sL_i^{k_i} + \sum_{j=0}^{k_i-1} L_i^j  a_{i,j}$, where 
$a_{i,j} \in \cR[\sM_1,\dots,\sM_r]$. 
Recall from that $\calF_N \subset \Trp$ is the $\cR$-span
of all monomials $\sM^\alpha \sL^\beta$ of total degree 
$|\alpha|+|\beta| \leq N$.  Then,
$\calF_N f$ is in the  $\cR$-span of $\sM^\alpha \sL^\beta f$ where
$|\alpha|+|\beta| \leq N$ and either $\beta=(\beta_1,\dots,\beta_r)$ satisfies
$\beta_i \leq k_i$ for $i=1,\dots,r$. The number of such monomials is
$O(N^r)$. Consequently, the dimension of $\Trp f$ is at most $r$, so
$f$ is $q$-holonomic with respect to $\Trp$. By Proposition \ref{prop.WWr}, 
$f$ is $q$-holonomic over $\Tr$.

For (b), suppose $\cL$ is a subset of 
$\{\sM_1,\dots,\sM_r,\sL_1,\dots, \sL_r\}$ with $r+1$ elements.  
Note that $d(\TrL)=r+1$. Suppose $f$ is  $q$-holonomic over $\Tr$. By 
Proposition \ref{prop.WWr}, $f$ is $q$-holonomic over $\Trp$ and hence 
$d(\Trp/\Ann_+(f))\le r$. Here $\Ann_+(f)=\Ann(f) \cap \Trp$.
It follows that $\Ann_+(f) \cap \TrL \neq \{0\}$, implying $f$ is strongly 
$\Tr$-finite.
\end{proof}


\begin{remark}
\label{rem.Dfinite}
The converse to (c) of Theorem~\ref{thm.Dfinite} does not hold.
Indeed, if $R(u) \in \cR(u)$ is a rational function in
$r$-variables $u=(u_1,\dots,u_r)$ and the function
$$
f: S_r(\BQ(q)), \qquad n \in \BZ^r \mapsto f(n)=R(q^n)
$$
is well-defined, then it is $\Tr$-finite. On the other hand, $f$
rarely satisfies the elimination property, hence it is not
$q$-holonomic in general. Concretely, 
C. Koutschan pointed out to us the following example:
\be
\label{eq.f2}
f: \BZ^2 \to \BQ(q), \qquad
f(n,k)=\frac{1}{q^n+q^k+1} \,.
\ee
It is obvious that $f$ is $\cT_2$-finite. On the other hand, 
$f(n,k)$ does not satisfy the elimination property for
$\{\sM_n,\sL_k,\sL_n\}$, hence it is not $q$-holonomic. 
To show this, assume the contrary. Then, there exists a nonzero operator
$$
P = \sum_{i,j} c_{i,j}(q^n,q) \sL_k^i \sL_n^j
$$
(for a finite sum)
where the $c_{i,j}$ are bivariate polynomials in $q$ and $q^n$. If $P$
annihilates $f$, this means:
$$
\sum_{i,j} \frac{c_{i,j}(q^n,q)}{q^{n+j} + q^{k+i} + 1} = 0 \,.
$$
Now observe that in no term can there be a cancellation, since the
numerator depends only on $q^n$. Next observe that the denominators of all
terms in the sum are pairwise coprime. Hence the expression on the
left-hand side is zero if and only if all $c_{i,j}$ are zero. This gives
a contradiction.
\end{remark}

We end this section by discussing a finite description of $q$-holonomic
functions, which is the core of an algorithmic description of $q$-holonomic
functions. For holonomic functions of continuous variables, the next theorem 
is known as the zero recognition problem, described in detail 
by Takayama~\cite[Sec.4]{Takayama}.

\begin{theorem}
\label{prop.finite}
Suppose $f\in S_r(V)$ is $q$-holonomic. Then there exists a finite set
$S \subset \BZ^r$ such that $f|_S$ uniquely determines $f$. In other words, 
if $g\in S_r(V)$ such that $\Ann(f)=\Ann(g)$ and $f|_S= g|_S$, then $f=g$.
\end{theorem}

\begin{proof}
We use induction on $r$. For $r=1$, this follows from 
Remark~\ref{rem.determines}. Suppose this holds for $r-1$. Since $f$
is strongly $\Tr$-finite, it follows that $f$ is annihilated by a nonzero
element $P=P(\sM_1,\sL_1,\dots,\sL_r) \in \Tr$. The $\sL$-exponents of $P$
is a finite subset of $\BN^r$. Recall the lexicographic total order 
$\alpha=(\alpha_1,\dots,\alpha_r) < \beta=(\beta_1,\dots,\beta_r)$ in 
$\BN^r$ (when $\alpha \neq \beta$) defined by the existence of $i_0$ such that 
$\alpha_i = \beta_i$ for $i < i_0$ and $\alpha_{i_0} < \beta_{i_0}$. 
Let $\sL^\alpha$ denote the leading term of $P$ in the lexicographic order.
Its coefficient is $p(q,q^{n_1})$ which is nonzero for all but finitely
many values of $n_1$. It follows by a secondary induction that the restriction 
of $f$ on $\BN^r$ is uniquely determined by its restriction on 
$\cup_{0 \leq k \leq r-1} (\BN^k \times J \times \BN^{r-1-k})$ for some finite
subset $J$ of $\BN$. Applying the same proof to the remaining $2^r-1$
orthants of $\BZ^r$ 
and enlarging $J$ accordingly (but keeping it finite), it follows 
that $f$ is uniquely determined by its restriction on 
$\cup_{0 \leq k \leq r-1} (\BZ^k \times J \times \BZ^{r-1-k})$ for some finite
subset $J$ of $\BZ$. Without loss of generality, we can assume that $J$
contains one element. Since $f$ is $q$-holonomic, it follows by parts (c)
and (d) of Theorem~\ref{r.sum} that its restriction on 
$\BZ^k \times J \times \BZ^{r-1-k}$ is $q$-holonomic too. The induction
hypothesis concludes the proof.
\end{proof}


\section{Algorithmic aspects}
\label{sec.algorithms}

From the very beginning, Zeilberger emphasized the algorithmic aspects of
his theory of holonomic functions, and a good place to start is the
book $A=B$~\cite{PWZ}. Algorithms and closure properties for the class of
$\Tr$-finite functions (of one or several variables) were developed and
implemented by several authors that include Chyzak, Kauers,
Salvy~\cite{Chyzak,Ka1} and especially Koutschan~\cite{Ko2}. A core-part of
those algorithms is elimination of $q$-commuting variables. The definition
of $q$-holonomic functions discussed in our paper is amenable to such
elimination, and we would encourage further implementations.


\section*{Acknowledgements}

The authors would like to thank Claude Sabbah and Christoph Koutschan for
useful conversations. The authors were supported in part by the US National
Science Foundation.

\bibliographystyle{hamsalpha}
\bibliography{biblio}
\end{document}